\documentclass[12pt,a4paper,reqno]{amsart}
\usepackage{amssymb,amsmath,amsthm}
\usepackage[T1]{fontenc}
\usepackage[dvips]{graphicx}
\usepackage{enumerate}

\usepackage[textures,bookmarks=false]{hyperref}
\usepackage{mathrsfs}

\theoremstyle{plain}
\newtheorem{theorem}{Theorem}
\newtheorem{maintheorem}{Theorem}

\newtheorem{lemma}{Lemma}
\newtheorem{proposition}{Proposition}
\newtheorem{remark}{Remark}

\newenvironment{definition}{{\bf Definition.}}{}

\def\sm{\setminus}

\def\R{\mathbb{R}}

\def\crit{\mbox{\rm Crit}}
\def\F{\mathcal{F}}

\def\eps{\varepsilon}
\def\phi{\varphi}
\def\le{\leqslant}
\def\ge{\geqslant}

\def\M{\mathcal{M}}

\def\dist{\text{dist}}

\def \cC {{\mathcal C}}
\def \cQ {{\mathcal Q}}

\newcommand{\sgn}{\operatorname{sgn}}

\def \RR {{\mathbb R}}

\def \cB {{\mathcal B}}
\def \cC {{\mathcal C}}

\def \cF {{\mathcal F}}

\def \cQ {{\mathcal Q}}

\def\flowclass{{\hat{\mathcal{F}}}}

\def\intclass{{\mathcal{F}}}
\def\Lyap{{\mathfrak{L}}}

\begin{document}

\title{Thermodynamic formalism for contracting Lorenz flows}
\author{Maria Jos\'e Pacifico}
\author{Mike Todd}
\date{\today}
\begin{thanks} {  M.J.P. was partially supported by
    CNPq-Brazil/FAPERJ-Cientista do Nosso Estado E-26/100588/2007/Pronex Dynamical Systems/Scuola Normale Superiore-Pisa.
MT was partially supported by FCT grant SFRH/BPD/26521/2006 and by FCT through CMUP}
\end{thanks}

\address{Maria Jos\'e Pacifico, Instituto de Matem\'atica,
Universidade Federal do Rio de Janeiro,
C. P. 68.530, 21.945-970 Rio de Janeiro, Brazil}
\email{pacifico@im.ufrj.br}

\address{Mike Todd, Centro de Matem\'atica da Universidade do Porto, Rua do Campo Alegre 687, 4169-007 Porto, Portugal\footnote{
{\bf Current address:}\\
Department of Mathematics and Statistics\\
Boston University\\
111 Cummington Street\\
Boston, MA 02215\\
USA }\\}
\email{mtodd@math.bu.edu}
\urladdr{http://math.bu.edu/people/mtodd/}

\begin{abstract}
We study the expansion properties of the contracting Lorenz flow introduced by Rovella via thermodynamic formalism.   Specifically, we prove the existence of an equilibrium state
for the natural potential $\hat\phi_t(x,y, z):=-t\log J_{(x, y, z)}^{cu}$ for the contracting Lorenz flow and for $t$ in an interval containing $[0,1]$. We also analyse the Lyapunov spectrum of the flow in terms of the pressure.
\end{abstract}



\maketitle

\section{Introduction}
\label{sec:intro}

The Lorenz flow \cite{Lorenz} is one of the key examples in the theory of dynamical systems due to the chaotic nature
of its dynamics, its robustness and its  connection with hydrodynamical systems.
 The Lorenz attractor, a `strange attractor' with a characteristic butterfly shape, has extremely rich dynamical properties
which have been studied from a variety of viewpoints: topological, geometric and statistical, see \cite{Sp, ArPa}.
Part of the reason for the richness of the Lorenz flow is the fact that it has an equilibrium, i.e. a fixed point,  accumulated by
regular orbits (orbits through points where the corresponding vector field does not vanish) which prevents the flow from being uniformly hyperbolic.
Indeed it is one of the motivating examples in the study of non-uniformly hyperbolic dynamical systems \cite{MorPacPu}.  It is also robust in the sense that nearby flows also possess strange attractors with similar properties.  The Lorenz equations can be studied using geometric models of the Lorenz flow, see \cite{Afr_etc, GuckWi}.  It was shown by Tucker \cite{Tucker} that the Lorenz equations do indeed support a geometric Lorenz flow.

The classical geometric Lorenz flow is expanding.  This corresponds to the Lyapunov exponents at the origin, $\lambda_s$ and $\lambda_u$, the stable and unstable exponents respectively, having $\lambda_u+\lambda_s>0$.
A Rovella-like attractor  \cite{Rovella} is the maximal invariant set of a
geometric flow whose construction is very similar to the one
that gives the geometric Lorenz attractor,
\cite{Afr_etc, GuckWi,ArPa}, except for the fact that the
eigenvalue relation $\lambda_u + \lambda_s > 0$ there is
replaced by $\lambda_u + \lambda_s < 0$.
As in the case for the geometric Lorenz attractor,
a Rovella attractor has a global cross section:
a line $\Gamma \subset \R^3$, and a first return map
$\tilde f$ defined on $\R^3 \setminus \Gamma$ that preserves
a one-dimensional foliation which is contracted under
the action of $\tilde f$. Thus, as in the case of the geometrical
model for the Lorenz flow, it is possible to study the
dynamics of a Rovella flow through a $1$-dimensional map
obtained quotienting though the leaves of this contracting
foliation.
Unlike the one-dimensional
Lorenz map obtained from the usual construction of the
geometric Lorenz attractor, a one-dimensional Rovella map
has a criticality at the origin, caused by the eigenvalue
relation $\lambda_u + \lambda_s<0$.
In Figure~\ref{fig-Lorenz3D} below we present some possible
``Rovella one-dimensional maps'' obtained through
quotienting out the stable direction of the return map to
the global cross-section of the attractor. In Section
\ref{sec:constructionRovella} we explain this procedure.

In this paper we will study the geometric model of the Rovella-like attractor
from the point of view of thermodynamic formalism.  This theory studies the multifractal properties of the system (see \cite{Pesbook}), providing precise characterisations of the dimension theory, as well as giving insight into the statistical properties of the system.  In the study of thermodynamic formalism, one takes a dynamical system $(f_s)_s:X\to X$ and a relevant potential $\phi:X\to \R$ and studies the statistical properties of the system through the properties of the pressure and the equilibrium states of the triple $(X, (f_s)_s, \phi)$.  This theory was developed for hyperbolic dynamical systems by Sinai, Ruelle and Bowen
\cite{Sinai, Ruelle, Bowen} in the context of H\"older potentials on hyperbolic dynamical systems, and has mainly been applied to Axiom A systems and Anosov
diffeomorphisms, see e.g.\ \cite{Baladi,Kellbook}.

The potentials which tell us most about the system involve the Jacobean of the flow/map.  For discrete smooth conformal systems one would consider $\phi=\log |Df|$.  Knowledge of the pressure and equilibrium states with respect to the family $t\phi$, the family of `natural/geometric' potentials, give us very fine information on the expansion properties of the system.  This is the Lyapunov spectrum.  For discrete uniformly hyperbolic systems the theory of thermodynamic formalism is already fairly well developed, see for example \cite{Urb_surv, Ols}.  However, for discrete conformal non-uniformly hyperbolic dynamical systems the theory is currently seeing a lot of activity, for example \cite{Nak,PolWei, GelRa, BTeqgen, BTeqnat, GelRaPrz, PrzR-L, IomTeq, IomTlyap}.  In the case of flows, thermodynamic formalism has been studied in the hyperbolic case in \cite{Bowen, Wei, BarSau_hyp_flow, PesSad, BarSau_var_flow, Cher}. In the non-uniformly hyperbolic case, the main contribution was made by Barreira and Iommi \cite{BarIom} who considered thermodynamic formalism for suspension flows over countable Markov shifts.

To understand the Jacobean for the Lorenz flow we note that the tangent space can be split into three directions: the flow direction, which has neutral expansion, the expanding/unstable direction and the contracting/stable direction.  The study of the Lyapunov spectrum in the Rovella flow case is particularly complicated since, in contrast to the expanding Lorenz case, we have to deal with points where a derivative is zero.

The interesting part of the dynamics is in the expanding part of the attractor, so we consider the Jacobean restricted to the expanding direction.  This situation can be modelled by a suspension flow over a countable Markov shift as in \cite{BarIom}, but our approach uses a simpler suspension flow allied to the results of Iommi and Todd \cite{IomTeq, IomTlyap}.  (Note that in \cite{IomTeq, IomTlyap} a countable Markov shift was used to produce the equilibrium states and information on the Lyapunov spectra.)  Our analysis captures the points which are typical for the physical measure as well as for many other points captured by nearby measures.  As mentioned above, we use the common approach (see \cite{MorPacPu, Metz_SRB, MetzMor, HolMel, ArPacPuVi, GalPac}) of analysing Lorenz-like flows by taking Poincar\'e sections in such a way that we obtain a one-dimensional map.   Note that our results hold for a larger class of maps than just the Rovella type of Lorenz flow.  We consider flows which have a Poincar\'e section with the dynamics of maps considered in the appendix of \cite{IomTeq}.

\section{The main results}
\label{sec:main res}

As sketched in the introduction, we prove the existence of an equilibrium state for the potential $\hat\phi_t(x,y, z):=-t\log J_{(x, y, z)}^{cu}; \,\,t\ge 0,$ for maps $\hat f=(\hat f_s)_s$ in a class $\flowclass$ of flows that includes a contracting Lorenz flow introduced in \cite{Rovella}.  This is the natural potential to consider for these maps.  Indeed, analysis of this potential also allows us to express the Lyapunov spectrum of the flow in terms of the pressure. Recall that a \emph{contracting Lorenz flow $\hat f$} is a flow with a unique singularity
at the origin $0$, defined in a compact neighbourhood
$\cC$ of  $0$ satisfying the following properties:
\newcounter{Lcount}
\begin{list}{(\arabic{Lcount})}
{\usecounter{Lcount} \itemsep 1.0mm \topsep 0.0mm \leftmargin=7mm}
\item the restriction of the flow to a small neighbourhood $\cQ \subset \overline{\cQ}\subset \cC$
is a linear flow $L$ with a unique singularity at $0$,
\item the eigenvalues $\lambda_i,\,\, 1\le i \le  3,$ of $DL(0)$ are all real and satisfy
$\lambda_2< \lambda_3 < 0 < -\lambda_3 < \lambda_3$.
\end{list}

It was proved in \cite{Rovella} that, under certain additional conditions,
the maximal positive $\hat f$-invariant set $\Lambda\subset \cQ$ is a transitive attractor.

In order to give our main results for these systems we first need to introduce some basic notions from thermodynamic formalism.
For references on the general theory, see for example \cite{Bowen, Kellbook, Pesbook, Cher}.

\subsection{Thermodynamic formalism}
\label{subsec:therm}

We begin by giving definitions for discrete time dynamical systems $f:X \to X$, and will then generalise to the flow case.  We let $$\M=\M(f):=\left\{\text{measures }\mu: \mu\circ f^{-1}=\mu \text{ and } \mu(X)=1\right\}.$$  Given a \emph{potential} $\phi: X \to [-\infty, \infty]$, the \emph{pressure} of $\phi$ with respect to $f$ is defined as
\begin{equation*}
P(\phi) =P(f,\phi):= \sup \left\{ h(\mu) + \int \phi~d\mu :  \mu \in \M \textrm{ and } - \int \phi~d\mu < \infty\right\},
\end{equation*}
where $ h(\mu)$ denotes the measure theoretic entropy of $f$ with respect to $\mu$.   As in \cite{Kellbook}, the quantity $h(\mu)+\int\phi~d\mu$ is
referred to as the \emph{free energy} of $\mu$ with respect to $(X,f,\phi)$.
A measure $\mu\in \M$ maximising the free energy, i.e. with $h(\mu)+\int\phi~d\mu=P(\phi)$, is called an \emph{equilibrium state}.

Similarly for a flow $\hat f$, we define the set of $\hat f$-invariant measures as
$$\M=\M(\hat f):=\left\{\text{measures } \hat\mu : \hat\mu(\hat f_s^{-1}(A))=\hat\mu(A) \text{ for all } s\ge 0 \text{ and } \hat\mu(\hat X)=1\right\}.$$
Moreover, for a potential $\hat\varphi:\hat X \to \R$, the pressure of $(\hat X, \hat f,\hat\varphi)$ is defined as
$$P(\hat f,\hat\varphi):=\sup\left\{h(\hat f, \hat\mu)+\int\hat\varphi~d\hat\mu:\hat\mu\in \M(\hat f) \text{ and } -\int\hat\varphi~d\hat\mu<\infty\right\}.$$
(For more details of the entropy of flows, see Section~\ref{sec:therm for flows}, in particular \eqref{eq:ent of flow}.)

For a flow $(\hat f_s)_s:\hat X\to \hat X$ in our class $\flowclass$, as in \cite{ArPacPuVi, MetzMor, MorPacPu} at each point $(x,y,z)\in \R^3$, the tangent space for the flow $\hat f$ has a splitting $E_x^{cu}\oplus E_x^s$ where $E_x^s$ is tangent to the stable direction and $E_x^{cu}$ is tangent to the centre unstable direction (see Section~\ref{sec:constructionRovella} for more details).
We are interested in the potential
\begin{equation}
\hat\phi_t(x,y, z):=-t\log J_{(x, y, z)}^{cu}\label{eq:hat phi t}
\end{equation}
which is the Jacobean of the differential in the centre unstable direction at the point $(x, y, z)$.
This potential gives rise to a natural class of equilibrium states, which can be seen as selecting out the sets in $\R^3$ with different rates of asymptotic expansion by the flow (see also Theorem~\ref{thm:lyap spec flow}).
For the following theorem, our first main theorem for contracting Lorenz flows, we consider this potential for $t\in (t^-, t^+)$.  The values of $t^-\le 0$ and $t^+\ge 1$ are given below in \eqref{eq:t plus minus}.

\begin{maintheorem}
Let $\hat f\in \flowclass$.  Then for all $t\in (t^-, t^+)$, there is an equilibrium state
$\hat\mu_t$ for $\hat\phi_t(x,y, z)$.
\label{thm:eq for flow}
\end{maintheorem}

Given a potential $\hat\phi:\hat I \to \R$, and $\alpha\in \R$, let
$$K^{\hat\phi}(\alpha):=\left\{(x,y,z)\in \hat I: \lim_{u\to \infty}\frac1u\int_0^u \varphi(\hat f_s(x,y,z))~ds = \alpha\right\}$$
and
$$K^{\hat\phi'}:=\left\{(x,y,z)\in \hat I: \lim_{u\to \infty}\frac1u\int_0^u \varphi(\hat f_s(x,y,z))~ds \text{ does not exist}\right\}.$$

In our second main theorem for contracting Lorenz flows, we take
$K(\alpha):=K^{\log J^{cu}}(\alpha)$.  The \emph{Lyapunov spectrum} of $(\hat I, \hat f)$ is the map
$$\alpha\mapsto \Lyap_{\hat f}(\alpha):= \dim_H(K(\alpha)\cap\Lambda),$$
where $\dim_H$ denotes the Hausdorff dimension of a set.

In our analysis of $\Lyap_{\hat f}$, we will use the potentials $\hat\phi_t$, for certain parameters $t\in \R$, and their equilibrium states.  The flows we consider and the potentials $\hat\phi_t$  have a natural relation with piecewise $C^2$ maps $f$ on an interval and the natural potentials \begin{equation}
\phi_t(x):=-t\log|Df(x)|. \label{eq:phi t}
\end{equation}
Often it can be shown that an equilibrium state for one such potential $\phi_1$ is an absolutely continuous invariant probability measure (acip) $\mu_{ac}$.

Defining, for a measure $\mu\in \M$, the \emph{Lyapunov exponent} of $(I, f,\mu)$ by
$$\lambda(\mu):=\int\log|Df|~d\mu,$$
any equilibrium state $\mu_t$ for $\phi_t$ therefore satisfies $$h(\mu_t)-t\lambda(\mu_t)=P(\phi_t).$$
We also define the pressure function:
$$p(t):=P(-t\log|Df|).$$

In the following theorem, we give a relation between Lyapunov spectrum and the pressure function on a certain domain $(\alpha_1,\alpha_2)\subset \R$ which is defined later in \eqref{eq:lyap dom}.  Note that in general the interval $(\alpha_1, \alpha_2]$ contains the Lyapunov exponents of both the SRB measure and the measure of maximal entropy.
We restrict our analysis to a subset of maps $\flowclass_{ac} \subset \flowclass$, which will be defined below.

\begin{maintheorem}
Let $\hat f\in \flowclass_{ac}$.  Then for all $\alpha\in (\alpha_1,\alpha_2)$, the Lyapunov spectrum satisfies the following relation
\begin{equation*}
\Lyap_{\hat f}(\alpha)-2 = \frac{1}{\alpha} \inf_{t \in \mathbb{R}}\left(p(t) + t \alpha\right) =
 \frac{1}{\alpha} \left( p(t) +t_{\alpha} \alpha \right).
\end{equation*}
where $t_\alpha$ is such that $Dp(t_\alpha) = -\alpha$.
\label{thm:lyap spec flow}
\end{maintheorem}

\begin{remark}
Note that it can be shown that for potentials $\phi_t$ and $\hat\phi_t$, with corresponding equilibrium states $\mu_t$ and $\hat\mu_t$, we have
$\Lyap_{\hat f}(\alpha)-2=
 \frac{h(\mu_{t_\alpha})}{\alpha} = \frac{h(\hat\mu_{t_\alpha})}{\alpha}$ for $\alpha$ and $t_\alpha$ as in Theorem~\ref{thm:lyap spec flow}.  Moreover,  $\hat\mu_{t_\alpha}(\R^3\sm \Lyap_{\hat f}(\alpha))=0$.
\end{remark}

We will prove Theorems~\ref{thm:eq for flow} and \ref{thm:lyap spec flow} by first reducing the study of maps in $\flowclass$ to a class of maps on the unit square and then to a further class of maps on the unit interval.  This is explained in the following two sections.

\section{Construction of a Rovella flow}
\label{sec:constructionRovella}

In this section we will consider a class of three dimensional flows which will be defined axiomatically.
To show that these axioms are verified in the geometric contracting Lorenz models we give a detailed
construction of this model.

We first analyse the dynamics in a neighbourhood of the
singularity at the origin, and then we complete the flow, imitating the butterfly shape of the original Lorenz flow.

We start with a linear system $(\dot x, \dot y, \dot
z)=(\lambda_1 x,\lambda_2 y, \lambda_3 z)$, with $\lambda_i$, $1\le
i\le 3$ satisfying the relation
\begin{equation}
 \label{eq:eigenvalues}
-\lambda_2 > - \lambda_3> \lambda_1 > 0,\quad \beta>\ell +3,\quad \beta= - \frac{\lambda_2}{\lambda_1},\,\,
\ell=- \frac{\lambda_3}{\lambda_1}.
\end{equation}

This vector field
will be considered in the cube $[-1,1]^3$ containing the origin $(0,0,0)$.

 For this linear flow, the trajectories are given by
\begin{align}\label{eq:LinearLorenz}
\hat f_s(x_0,y_0,z_0)=
(x_0e^{\lambda_1s}, y_0e^{\lambda_2s}, z_0e^{\lambda_3s}),
\end{align}
where  $(x_0, y_0, z_0)\in\RR^3$ is an
arbitrary initial point near $p=(0,0,0)$.

Now let $\Sigma=\big\{ (x,y,1) : |x|\le
{\scriptstyle{1/2}},\quad |y|\le{\scriptstyle{1/2}}\big\}$ and consider \begin{align*}
\Sigma^-&=\big\{ (x,y,1)\in \Sigma : x<0 \big\},&
\qquad
\Sigma^+&=\big\{ (x,y,1)\in \Sigma : x>0 \big\} \text{ and}
\\
\Sigma^*&=\Sigma^-\cup \Sigma^+=\Sigma\setminus\Gamma,& \text{where}\quad\quad\quad
\Gamma&=\big\{(x,y,1)\in \tilde{I} : x=0 \big\}.
\end{align*}

$\Sigma$ is a transverse section to the linear flow and
every trajectory crosses $\Sigma$ in the direction of
the negative $z$ axis.

Consider also
$\tilde{\Sigma}=\{ (x,y,z) : |x|=1\}=\tilde{\Sigma}^-\cup {\tilde{\Sigma}}^+$
with ${\tilde{\Sigma}}^{\pm}=\{ (x,y,z): x=\pm 1\}$.  For each $(x_0,y_0,1)\in \Sigma^*$ the time $s$
such that $\hat f_{s}(x_0,y_0,1)\in\tilde{\Sigma}$ is given by
\begin{equation}\label{eq:tempo}
s(x_0)=-\frac{1}{\lambda_1}\log{|x_0|}
\end{equation}
which depends on $x_0\in \tilde\Sigma^*$ only and is such that $s(x_0)\to+\infty$ when $x_0\to0$.

Hence, using (\ref{eq:tempo}), we get (where $\sgn(x)=x/|x|$ for $x\neq0$)
\begin{align*}
\hat f_{s(x_0)}(x_0,y_0,1)=
\big( \sgn(x_0),  y_0e^{\lambda_2\cdot s(x_0)}, e^{\lambda_3\cdot s(x_0)}\big)
=
\big( \sgn(x_0),
y_0|x_0|^{\beta},
|x_0|^{\ell}\big).
\end{align*}

Consider
 $L:\Sigma^*\to\ {\tilde{\Sigma}}^{\pm}$  defined by
\begin{equation}\label{L}
L(x,y,1)=\big(\sgn(x),
y|x|^\beta,|x|^\ell\big).
\end{equation}

 Clearly each segment $\Sigma^*\cap\{x=x_0\}$ is
taken by $L$ to another segment $\tilde{\Sigma}^\pm\cap\{z=z_0\}$ as
sketched in Figure~\ref{L3Dcusp}.

\begin{figure}[h]
\begin{center}
\includegraphics[width=7cm]{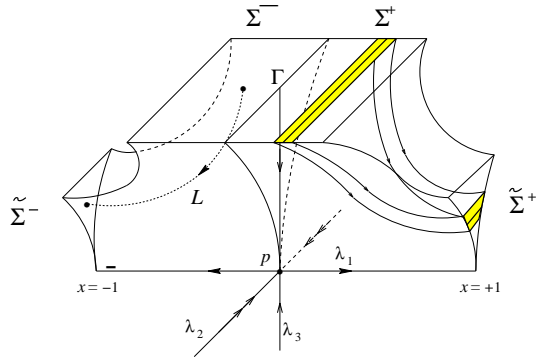}
\end{center}
\caption{\label{L3Dcusp}Behaviour near the origin.}
\end{figure}
It is easy to see that $L(\Sigma^\pm)$ has the shape of a cusp
triangle with a vertex $(\pm 1,0,0)$, a
cusp point at the boundary of the triangle.
Since $L$ is a linear flow, it preserves the vertical
foliation $\cF^s$ of $\Sigma$ whose leaves are given by the lines $x=x_0$.
We shall further assume that
 $L(\Sigma^\pm)$ are uniformly compressed in the $y$-direction.

\subsection{The random turns around the origin}
\label{rotacao}

To imitate the random turns of a regular orbit around the origin and obtain
a butterfly shape for our flow,  we proceed as follows.

Recall that the fixed point $p$ at the origin is hyperbolic and so its
stable $W^s(p)$ and unstable $W^u(p)$ manifolds are well defined, \cite{PaldM}.
Observe that $W^u(p)$ has dimension one and so it has two branches,
$W^{u,\pm}(p)$ and $W^u(p)=W^{u,+}(p)\cup\{p\}\cup W^{u,-}(p)$.

The sets $\tilde{\Sigma}^\pm$ should return to the cross section $\Sigma$ through a
flow described by a suitable composition of a rotation $R_\pm$, an expansion $E_{\pm\theta}$ and a translation $T_\pm$.

The rotation $R_\pm$ has axis parallel to the
$y$-direction, which is orthogonal to the $x$-direction (which is parallel to the local branches $W^{u,\pm}(p)$). More
precisely is such that $(x,y,z)\in \tilde{\Sigma}^\pm$, then
\begin{align}
 \label{derivadadeR}
R_\pm(x,y,z)=
\left(
\begin{array}{cccc}
      0 & 0 & \pm 1       \\
 0 & 1 & 0\\
\pm 1 & 0 & 0
\end{array}
\right).
\end{align}
The expansion occurs only along the $x$-direction, so, the matrix of
$E_{\theta}$ is given by

\begin{align}
 \label{dilatacao}
E_{\pm\rho}(x,y,z)=
\left(
\begin{array}{cccc}
    \rho   & 0 & 0       \\
 0 & 1 & 0\\
0 & 0 & 1
\end{array}
\right)
\end{align}
with $\rho \cdot ({\frac{1}{2}} )^\ell<1 $.
This condition is to ensure that the image of the resulting map is contained in $\Sigma$.

The translation $T_\pm:\RR^3\to \RR^3$ is chosen such that the unstable direction starting
from the origin is sent to the boundary of $\Sigma $ and the image
of both $\tilde\Sigma^\pm$ are disjoint.
 These transformations $R_\pm , E_{\pm\rho}, T_\pm $ take line
segments $\tilde\Sigma^\pm\cap\{z=z_0\}$ into line segments
$\Sigma\cap\{x=x_1\}$, and so does
the composition $T_\pm\circ E_{\pm\rho}\circ R_\pm$.

This composition of linear maps describes a vector field in a region outside
$[-1,1]^3=\hat{I}$ in the sense that one can use the above matrices to define a vector field $V$ such that the time one map of the associated flow
realises $T_\pm\circ E_{\pm\rho}\circ R_\pm$ as a map $\tilde\Sigma^\pm\to \Sigma$.  This will not be explicit here, since the choice of the vector field  is not really important for our purposes.

The above construction allows us to describe, for each $s\in\mathbb{R}$, the orbit $f_s(x)$  of each point $x \in \tilde{I}$: the orbit will
start following the linear field until  $\widetilde{\Sigma}^\pm$ and then it will follow $V$ coming back to $\Sigma$ and so on.
 Let us denote by ${\cB}=\{\hat f_s(x),x\in \Sigma, s\in \mathbb{R}^+\} $ the set where this flow acts.
The geometric contracting Lorenz flow is then the couple $({\cB}, \hat f_s )$ defined in this way.

\begin{figure}[htpb]
\centerline{\includegraphics[width=5cm]{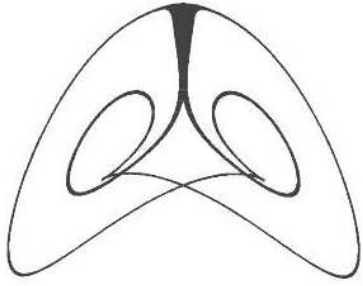}}
\caption{\label{L1D}{A Rovella flow}}
\end{figure}

The Poincar\'e first return map will thus be defined by
$\tilde f:\Sigma^*\to \Sigma$ as
\begin{equation}
 \label{F}
\tilde f(x,y)=\left\{
\begin{array}{ccc}
 T_+\circ E_{+\rho}\circ R_+\circ {L}(x,y,1) & \mbox{for }\, x>0\\
T_-\circ E_{-\rho}\circ R_-\circ {L}(x,y,1) &  \mbox{for }\, x < 0
\end{array}
\right.
\end{equation}

The combined effects of $T_\pm\circ R_\pm$ and ${L}$ on lines implies that the foliation $\cF^s$ of $\Sigma$ given by the lines $\Sigma\cap\{x=x_0\}$ is invariant under the return map. In other words, we have
\begin{itemize}
\item[$(\star)$] {\em for any given
leaf $\gamma$ of $\cF^s$, its image $F(\gamma)$ is
contained in a leaf of $\cF^s$, and the condition $\beta>\ell +3$ guarantees that $\cF^s$ is a $C^3$-foliation.}
\end{itemize}

\subsection{An expression for the first return map}
\label{sec:exprassaodeF}

Combining equations (\ref{L}) with the effect of the rotation composed with
the expansion and the translation,
we obtain that $\tilde f$ must have the form
\begin{equation*}\label{eq:Fgrande}
\tilde f(x,y)=\big(f_{Ro}(x),g_{Ro}(x,y)\big)
\end{equation*}
where
 $f_{Ro}:I\setminus\{0\}\to I$ and
$g_{Ro}:(I\setminus\{0\})\times I\to I$ are given by

\begin{equation}
\label{eq:fLo}
f_{Ro}(x)=
\begin{cases}
 f_1(x^\ell) &  \text{if } x < 0,       \\
f_0(x^\ell) &  \text{if } x > 0,
\end{cases}
\quad \mbox{with $f_i =(-1)^{i}\rho \cdot x+d_i, i\in\{0,1\}$, and }
\end{equation}

\begin{equation*}
g_{Ro}(x,y)=
\begin{cases}
g_1(x^\ell,y\cdot x^\beta) &  \text{if } x < 0,       \\
g_0(x^\ell,y\cdot x^\beta) &  \text{if } x > 0,
\end{cases}
\,\,
\end{equation*}
where  $g_1|I^-\times I\to I $ and $g_0|I^+\times I\to I$ are
suitable affine maps. Here $I^-=(-1/2,0)$, $I^+=(0,1/2)$.
Note that conditions (f1)-(f5) below determine the precise form of the constants $d_i$.

\subsubsection{Properties of the  map $g_{Ro}$}
\label{sec:asegundacoordenada}

Observe that by construction, $g_{Ro}$ in equation (\ref{F}) is piecewise $ C^3$.  Moreover, we have the following bounds on
its partial derivatives:

\begin{enumerate}
 \item[(a)] For all $(x,y)\in\Sigma^*, x> 0$, we have
${\partial_y} g_{Ro}(x,y)=  x^\beta$. As $\beta>1$, $|x|\le 1/2$,
there is $0<\lambda<1$ such that
\begin{equation}
 \label{gy}
|{\partial_y} g_{Ro}| < \lambda.
\end{equation}
The same bound works for $x < 0$.

\item[(b)] For all $(x,y)\in\Sigma^*, x \neq 0 $, we have
${\partial_x} g_{Ro}(x,y)=\beta\cdot x^{\beta-\ell}$.
As $\beta-\ell > 3$  and $|x|\le 1/2$,  we get
\begin{equation}
 \label{gx}
|\partial_x g_{Ro}| < \infty.
\end{equation}
\end{enumerate}
Item (a) above implies that the map $\tilde f=(f_{Ro},g_{Ro})$ is uniformly
contracting on the leaves of the foliation $\cF^s$:
there is $ C >0$ such that

\begin{itemize}
\item[$(\star \star)$]  if $\gamma$ is a leaf of $\cF^s$ and $x,y\in\gamma$, then $\dist\big(\tilde f^n(x),\tilde f^n(y)\big)\le {\lambda^ n}\cdot C\cdot  \dist(x,y)$
\end{itemize}
where $ \lambda$ can be chosen as the one given by equation  (\ref{gy}).

\subsubsection{Properties of the one-dimensional map $f_{Ro}$}
\label{sec:propert-one-dimens}

Next we outline the main features of $f_{Ro}$.

The following properties are easily implied from the
construction of $\hat f$:
\begin{enumerate}
\item[(f1)] By equation (\ref{eq:fLo}) and the way $T_\pm$ is defined,
$f_{Ro}$ is discontinuous at $x=0$. The lateral limits
  $f_{Ro}(0^\pm)$ do exist,  $f_{Ro}(0^\pm)=\pm\frac12$,

\item[(f2)] $f_{Ro}$ is $C^3$ on $I\setminus\{0\}$. As $\beta > \ell + 3$, we get
$\lim_{x\to 0^+}Df_{Ro}(x)=0=\lim_{x\to 0^-}Df_{Ro}(x)$, and the order of $f_{Ro}$ at
$x=0$ is $\ell - 1 > 0$.

By the convexity properties of $f_{Ro}$ we then obtain that
\begin{equation*}
 \label{eq:derivadamaiorqueum}
Df_{Ro}(x) > 0 \quad \quad \mbox{for all}\quad \quad x \in I\setminus\{0\}.
\end{equation*}

\item[(f3)] $\max_{x>0} Df_{Ro}(x)=Df_{Ro}(1),\quad \max_{x<0} Df_{Ro}(x)=Df_{Ro}(- 1).$

\item[(f4)] $-1$ and $1$ are pre-periodic repelling points for $f_{Ro}$.

\item[(f5)] $f_{Ro}$ has negative Schwarzian derivative: $Sf_{Ro} <  a  < 0.$
\end{enumerate}

We say that a map of the interval $f:I\to I$ is a \emph{Rovella map} if it satisfies the properties (f1)--(f5) above.
We denote this class of maps by $\cF_R$.
We refer to a Rovella map which is topologically conjugate to the doubling map as a \emph{full Rovella map}.

\begin{figure}[h]
  \centering
  \includegraphics[width=10cm]{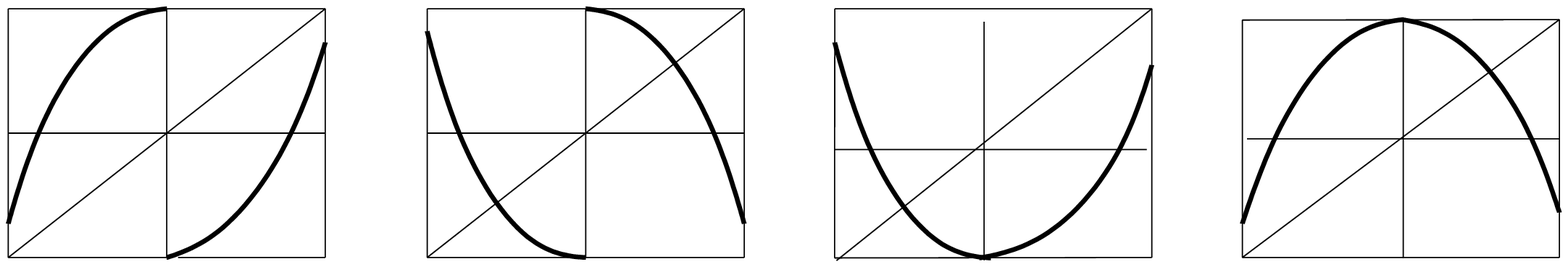}
  \caption{\label{fig-Lorenz3D} Possible Rovella maps}
\end{figure}

\subsubsection{A Rovella attractor is partially hyperbolic}
\label{sec:Rovparthyp}
A compact invariant set $\Lambda \subset M$ is {\em{partially
hyperbolic}} if the tangent bundle $T_\Lambda M$ splits into
a continuous sum of sub-bundles $E\oplus F$, $D\hat f_s$-invariant, with $E$
uniformly contracting, $F$ contains the flow-direction $[\hat f]$
and there are  $0< \lambda  < 1$ and $c > 0$ such that
that for all $s > 0$ and each $x \in \Lambda$
\begin{equation}\label{eq.domination}
\left\|D\hat f_s \mid E^s_x\right\| \cdot
\left\|D\hat f_{-s} \mid E^{cu}_{\hat f_s(x)}\right\| < c \, \lambda^s.
\end{equation}

It follows from the construction and condition (\ref{eq:eigenvalues}) on the eigenvalues
at the origin, that a Rovella attractor is partially hyperbolic.
In particular, besides the existence of the stable (uniformly contracting)
foliation $\cF^s$, there is a centre-unstable $C^1$ foliation $\cF^{cu}$.

\subsubsection{Projection to the interval}
\textit{Notation:} From here on it will often be convenient to use the notation $\tilde I = \Sigma$
(recall $I=[-1/2, 1/2]$) and $\hat I=\cB$, the domain of the Rovella flow constructed above.

The intersection of the foliations $\cF^s$ and $\cF^{cu}$ with the cross section $\tilde I$ induce a coordinate system $(x,y)$ on $\tilde I$, i.e. any point in $\tilde I$ can be expressed as $(x, y)$ where for all small $\eps$, all points $(x+\eps', y)$ for
$|\eps'|<\eps$ are in the same unstable leaf as $(x,y)$, and similarly $(x, y+\eps')$
are in the same stable leaf as $(x,y)$.

We define the map $\iota: \tilde I \to I$ as $\iota(x,y)=x$.  Since our map $\tilde f$ preserves the stable foliation, $(I, f)$ is a factor of $(\tilde I, \tilde f)$.  That is, $f\circ \iota= \iota\circ \tilde f$. To see this, let $(x, y)\in \tilde I$ and suppose that $(x', y')\in \tilde I$ is such that $\tilde f(x, y)=(x', y')$.  From the definition of $f$, we have $x'=f(x)$.  Then we compute $$f\circ\iota(x,y)= f(x)=\iota(f(x), y')=\iota\circ \tilde f(x,y).$$

\section{$C^2$ cusp maps}

As in the previous section, given a Rovella map $\hat f$, if we take Poincar\'e sections twice then the study of the flow reduces to the study of one-dimensional maps.  We will shortly define a wider class of one-dimensional maps which contains this class (and so the corresponding class of flows contains Rovella flows).  First we define the left and right derivatives of a map $f:A \to \R$ for $x\in A$ where $A\subset \R$ as
$$D^-f(x):=\lim_{y \nearrow x}\frac{f(x)-f(y)}{x-y} \text{ and }
D^+f(x):=\lim_{y \searrow x}\frac{f(x)-f(y)}{x-y}$$
respectively.

\begin{definition}
$f:\cup_jI_j\to I$ is a \emph{non-singular cusp map} if there exist constants $C,\alpha>1$
and a finite  collection $\{I_j\}_j$ of disjoint open subintervals of $I$ such that
\begin{enumerate}

\item for all $x,y\in \overline{I_j}$ we have $|Df_j(x)-Df_j(y)|<C|x-y|^\alpha$;

\item $D^+f(a_j), \ D^-f(b_j)$ exist and are equal to 0.

\end{enumerate}
We denote the set of points $a_j, b_j$ by $\crit$.
\end{definition}

Dobbs \cite{Dobthes} considered maps of this type, although he also allowed the maps to have some types of singularities at the boundaries of $I_j$.  Note that the Lorenz-like maps considered in \cite{DiHoLu} are a subset of the cusp maps considered by Dobbs, but with extra expansion conditions.

\begin{remark}
Notice that if for some $j$, $b_j=a_{j+1}$, i.e. $I_j\cap I_{j+1}$ intersect, then $f$ may not continuously extend to a well defined function at the intersection point $b_j$, since the definition above would then allow $f$ to take either one or two values there.  So in the definition above, the value of $f_j(a_j)$ is taken to be $\lim_{x\searrow a_j}f_j(x)$ and $f_j(b_j)=\lim_{x\nearrow b_j}f_j(x)$, so for each $j$, $f_j$ is well defined on $\overline{I_j}$.
\label{rmk:boundaries}
\end{remark}

In this paper we will restrict to a particular subset of this class.
We let $\intclass$ be the class of non-singular cusp maps with
\begin{enumerate}
\item[(3)] negative Schwarzian (i.e. $1/\sqrt{|Df|}$ is convex on each $I_j$).
\end{enumerate}
This condition rules out the singularities considered by Dobbs.  Moreover, it is clear that the class $\intclass_R$ of Rovella maps described in the previous section is included in  $\intclass$. We let $\intclass_{ac}\subset\intclass$ denote the class of maps $f\in \intclass$ which have an acip $\mu_{ac}$ with positive Lyapunov exponent and which has density with respect to Lebesgue in $L^p$ for some $p>1$.  Note that maps in $\flowclass_R$ as well as the non-singular maps in \cite{DiHoLu} are in $\flowclass_{ac}$.  This can be derived for example from \cite[Lemma 2.2]{Collexval} and the exponential decay shown in \cite{DiHoLu, Metz_SRB}.

Next we introduce the class of flows we shall deal with:

\begin{definition}
\label{def:classflows}
The set $\flowclass$ is the class of flows on $\hat I$ which give rise to a Poincar\'e map on $\tilde I$ which is uniformly contracting in the vertical direction, the return time of $(x,y,1)$ is of order $-\log|x|$ (as in \eqref{eq:tempo}) and the map induced in the horizontal coordinate is in $\intclass$.  The set $\flowclass_{ac}$ is defined similarly.
\end{definition}

The study of the potential $\hat\phi_t$ as in \eqref{eq:hat phi t} for maps in $\flowclass$ reduces to the study of potentials $\phi_t$ as in \eqref{eq:phi t}.  In order to prove the existence of equilibrium states for these potentials we need to further restrict our class to maps with good expansion properties.  Note that our conditions are much weaker than those required for Rovella maps.

We define
$$\lambda_M=\lambda_M(f):=\sup\{\lambda(\mu):\mu\in \M\}, \ \lambda_m=\lambda_m(f):=\inf\{\lambda(\mu):\mu\in \M\}.$$
Then for $f\in \intclass$ we let
\begin{equation} \label{eq:t plus minus}
t^-:=\inf\{t:p(t)>-\lambda_M t\} \text{ and }
t^+:= \sup\{t:p(t)>-\lambda_m t\}.
\end{equation}

\begin{remark}
The arguments of \cite{Prz} can be adapted to show that if $f\in \intclass$ then $\lambda_m\ge 0$. This implies that $t^+> 0$.  If $f\in \intclass_{ac}$ then by definition the acip has positive Lyapunov exponent.  Therefore as in \cite[Theorem 3]{Dobcusp}, see also \cite[Theorem 3]{Led}, $\mu_{ac}$ is an equilibrium state for $-\log|Df|$ and moreover $t^+\ge 1$.

  Since $\lambda_M\le \sup_{x\in I}\log|Df(x)|<\infty$ for $f\in \intclass$, we also have $t^-<0$.
Note that if $t^->-\infty$ then $p$ is linear for all $t\le t^-$.
Similarly, if $t^+<\infty$ then $p$ is linear for all $t\ge t^+$.
\label{rmk:prz}
\end{remark}

The first theorem gives equilibrium states for our systems.  In the context of multimodal maps this theory was first considered in \cite{BrKell}, later extended for some cases by \cite{PeSe}, and then for more general cases in \cite{BTeqnat, BTeqgen} and in complete generality in \cite{IomTeq}.  The following theorem is proved in the appendix of \cite{IomTeq}.

\begin{theorem}
Let $f\in \intclass$.  Then for all $t\in (t^-, t^+)$ there is a unique equilibrium state $\mu_t$ for $\phi_t$.  Moreover,
\begin{enumerate}
\item $h(\mu_t)>0$;
\item the map $t\mapsto p(t)$ is $C^1$ in $(t^-, t^+)$;
\item if all $c\in\crit$ are not periodic or preperiodic then $t^-=-\infty$.
\end{enumerate}
\label{thm:eq states}
\end{theorem}

We next consider the Lyapunov spectrum.  For $\alpha \in \R$, we let
\begin{equation*}
J(\alpha):= \Big{\{} x \in I :  \lim_{n \to \infty} \frac{1}{n} \log |Df^n(x)| = \alpha \Big{\}}
\end{equation*}
and
\begin{equation*}
J':= \Big{\{} x \in I : \textrm{the limit} \lim_{n \to \infty} \frac{1}{n} \log |Df^n(x)| \textrm{ does not exist}  \Big{\}}.
\end{equation*}
The unit interval can be decomposed in the following way (the
\emph{multifractal decomposition}),
\begin{equation*}
 [0,1]= J' \cup \left( \cup_{\alpha} J(\alpha) \right).
\end{equation*}
As in Section~\ref{sec:main res}, the function that encodes this decomposition is called the \emph{multifractal spectrum of the Lyapunov exponents} and it is defined by
\begin{equation*}
\Lyap_f(\alpha):= \dim_H(J(\alpha)).
\end{equation*}
This function was studied by Weiss \cite{Wei} in the context of Axiom A maps.

As in the usual theory, if $p(t)$ is $C^1$ at $t\in \R$ and there exists an equilibrium state $\mu_t$ for $-t\log|Df|$ then $Dp(t)=-\lambda(\mu_t)$.
Let \begin{equation}\alpha_1:=D^+p(t^-) \quad \text{ and } \quad \alpha_2:=D^-p(t^+). \label{eq:lyap dom} \end{equation}
 The following is proved as in \cite{IomTlyap}.

\begin{theorem}
Let $f\in\intclass_{ac}$. Then for all $\alpha \in (\alpha_1, \alpha_2)$, the Lyapunov spectrum satisfies the following relation
\begin{equation*}
\Lyap_f(\alpha) = \frac{1}{\alpha} \inf_{t \in \mathbb{R}}\left(p(t) + t \alpha\right) =
 \frac{1}{\alpha} \left( p(t_\alpha) +t_{\alpha} \alpha \right) =\frac{h(\mu_{t_\alpha})}{\alpha},
\end{equation*}
where $t_\alpha$ is such that $Dp(t_\alpha) = -\alpha$.  Moreover, $\Lyap_f$ is $C^1$ in $(\alpha_1, \alpha_2)$.
\label{thm:lyap spec}
\end{theorem}

\begin{remark}
Note that in the case of full Rovella maps, as in the case of the quadratic Chebyshev map, $\alpha_1=\alpha_2$, so the above theorem is empty.  This is shown via the conjugacy to the doubling map. Moreover, for both of these maps there are only two possible Lyapunov exponents, one corresponding to the repelling fixed points and one corresponding to the acip.  The former corresponds to a set of Hausdorff dimension 0 and the latter to a set of Hausdorff dimension 1.
\end{remark}

\begin{remark}
As in Remark~\ref{rmk:prz}, if $f\in \intclass_{ac}$ then  $t^-<0$ and $t^+\ge 1$.  Hence the interval $(\alpha_1, \alpha_2)$ contains the interval $(\lambda(\mu_{ac}), \lambda(\mu_{max})]$ where $\mu_{ac}$ is the acip (the equilibrium state for $-t\log|Df|$ for $t=1$) and $\mu_{max}$ is the measure of maximal entropy (the equilibrium state for $-t\log|Df|$ for $t=0$).
\end{remark}

\section{Thermodynamics of flows}
\label{sec:therm for flows}

\subsection{Thermodynamics of suspension flows}

We will show that for certain natural potentials for the contracting Lorenz flow, we can prove an equivalent of Theorem~\ref{thm:eq states}.  Given the Lorenz flow $\hat f=(\hat f_s)_{s\ge 0}$ on $\hat I\subset \R^3$, as shown Section~\ref{sec:constructionRovella}, we can take a 2 dimensional Poincar\'e section $\tilde I \subset \R^3$ and get the first return map $\tilde f=\hat f_r:\tilde I \to \tilde I$ where $r$ is the return time of a point in $\tilde I$ to $\tilde I$.

We can treat the Lorenz flow as a semiflow over $\tilde I$.  We will describe the abstract setup for semiflows.  For more background on this general setup we refer to \cite{AmbKak} which describes the simple relation between the semiflow and the map on the base which the semiflow is taken over.  Much of what follows is very similar to thermodynamic formalism in the setting of Anosov flows, see \cite{Bowen, Cher}.
However, the singularity causes some difficulties, creating some non-uniform hyperbolicity.  For some information on such systems, but principally for SRB measures, see \cite{Vibook}.  For related recent work on the thermodynamics of semiflows over Countable Markov Shifts to prove our results see \cite{BarIom}.

Suppose that $f:X \to X$ is a dynamical system.  Let the \emph{roof function} $\check r:X \to [0,\infty)$ be a continuous function and consider the space:
$$\check X:=\{(x,s)\in X\times\R: 0\le s\le \check r(x)\}/\sim,$$
where $(x,\check r(x)) \sim (\sigma(x), 0)$ for every $x\in X$.  The \emph{suspension semiflow $\check f=(\check f_s)_{s\ge 0}$ over $f$ with roof function $\check r$} is defined as
$$\check f_s(x,u):=(x, u+s) \text{ if } u+s\in [0, \check r(x)].$$

The relevant class of measures here is
$$\M(f, \check r):=\left\{\mu\in \M(f):\int \check r~d\mu<\infty\right\}.$$
It is shown in \cite{AmbKak} that if $\mu$ is an $f$-invariant measure, possibly infinite, and $\int \check r~d\mu<\infty$ then the product measure $\mu\times m$, where $m$ is Lebesgue, is $\check f$-invariant.  Indeed when $\check r$ is bounded away from zero
there is a canonical identification between $\M(\check f)$ and $\M(f, r)$: the map
$\check\iota:\M(f, \check r) \to \M(\check f)$ given by
\begin{equation}\check\iota(\mu):=\frac{(\mu\times m)|_{\check X}}{(\mu \times m)(\check X)}\label{eq:iota}\end{equation} is a bijection.

The entropy of a flow $(\check X, \check f,\check\mu)$ can be defined by the metric entropy of the corresponding time 1 map.  Abramov \cite{Abra} proved that for semiflows, this is the same as
\begin{equation}
h(\check f, \check\mu)=\frac{h(f, \check\mu\circ\check\iota^{-1})}{\int r~d(\check\mu\circ\check\iota^{-1})}.\label{eq:ent of flow}
\end{equation}
We will take this definition.

Given a potential $\check\varphi:\check X \to \R$, we define the corresponding potential $\Delta_{\check\varphi}: X \to \R$ as
$$\Delta_{\check\varphi}(x):=\int_0^{r(x)} \check\varphi(x,s)~dt.$$
By the identification of $\M(\check f)$ and $\M(f, r)$, coupled with the Abramov formula, we can write
$$P(\check f,\check\varphi)=\sup\left\{\frac{1}{\int r~d\mu}\left(h(f, \mu)+\int\Delta_{\check\varphi}~d\mu\right):\mu\in \M(f, r) \text{ and } -\int\Delta_{\check\varphi}~d\mu<\infty\right\}.$$

\begin{remark}
If the underlying system $(X,f)$ is a countable Markov shift, under certain smoothness conditions on the potential, a Variational Principle for the pressure was proved in \cite{BarIom}.  We could extend that theory to the case of the Lorenz flow with the potentials given below.  However, since this isn't required to prove our results, we will not do the computations here.
\end{remark}

We now return to the map $\tilde f:\tilde I \to \tilde I$.
As in \cite{AmbKak}, there is a map between the suspension flow and the actual flow:
$$\check p: \check I=\left\{(x,s)\in \tilde I \times \R: 0 \le s\le \check r(x)\right\}/\sim \to \hat I.$$
We can define entropy of a measure $\hat\mu\in \M(\hat f)$ as $h(\hat\mu\circ\check p)$.  Similarly we can define the pressure $P(\hat f, \hat\phi)$ as $P(\check f, \check\phi)$.

\subsection{The relation between the one and two dimensional systems}

Later we will relate equilibrium states for $f$ with those for $\tilde f$. Doing this involves comparing the quantity free energies of measures for $f$ with those for $\tilde f$, so we will need a relation between $\M(f)$ and $\M(\tilde f)$.
We will use ideas from \cite{ArPacPuVi} to help with this.  Note that in that paper the authors considered the expansive rather than the contracting Lorenz flows we are concerned with here, but many of those ideas carry through to our case.
As in \cite[Corollary 5.2]{ArPacPuVi}, there is an injection from $\M(f)$ to $\M(\tilde f)$.  Moreover, clearly given $\tilde\mu\in \M(\tilde f)$ the measure $\tilde\mu\circ\iota^{-1}$ is $f$-invariant.  In the following lemma we show that in fact we have a bijection between $\M(f)$ and $\M(\tilde f)$.

\begin{lemma}
There is a bijection $\tilde p:\M(f) \to \M(\tilde f).$  Moreover, $\iota\circ\tilde p$ is the identity on $\M(f)$ and $\tilde p$ and $\tilde p^{-1}$ take ergodic measures to ergodic measures.
\label{lem:tilde p}
\end{lemma}

\begin{proof}
Given $x\in I$, we let $\xi_x:=\{(x,y):y\in I\}$.  This is a leaf of the $\F^s$ the stable foliation of $\tilde I$.  For any potential $\tilde\phi:\tilde I \to \R$, we define
$$\phi_-(x)=\inf_{y\in \xi_x}\tilde\phi(x,y) \text{ and } \phi_+(x)=\sup_{y\in \xi_x}\tilde\phi(x,y).$$
Following \cite[Section 5.1]{ArPacPuVi}, we can show that if $\mu\in \M_f$ then there is a unique measure $\tilde\mu\in \M(\tilde f)$ such that for any continuous function $\tilde\phi:\tilde I\to \R$, the limits
$$\lim_{n\to \infty}\int\left(\tilde\phi\circ\tilde f^n\right)_-~d\mu \text{ and } \lim_{n\to \infty}\int\left(\tilde\phi\circ\tilde f^n\right)_+~d\mu$$ exist, are equal, and coincide with $\int\tilde\phi~d\tilde\mu$. This determines the map $\tilde p:\M(f) \to \M(\tilde f)$, which is injective. We will show that it is in fact a bijection between $\M(f)$ and $\M(\tilde f)$.

The map $\iota$ gives us a natural way to get from $\M(\tilde f)$ to $\M(f)$.  The lemma will be proved if we can show that given $\tilde\mu\in \M(\tilde f)$, for the measure $\nu:=\tilde\mu\circ\iota^{-1}\in \M(f)$ we have $\tilde\nu=\tilde\mu$
(i.e. $\tilde p\circ\iota$ is the identity on $\M(\tilde f)$).

As in \cite[Corollary 5.2]{ArPacPuVi}, we have, for a continuous $\tilde\phi:\tilde I \to \R$ and $\tilde\nu:=\tilde p(\nu)$, \begin{align*}
\int\tilde\phi~d\tilde\nu & =\lim_{n\to \infty}\int\left(\tilde\phi\circ \tilde f^n \right)_-~d\nu= \lim_{n\to \infty}\int(\tilde\phi\circ \tilde f^n)_-~d(\tilde\mu\circ\iota^{-1})\\
&=\lim_{n\to \infty}\int\left(\inf_{y'\in I}\tilde\phi\circ \tilde f^n(x,y')\right)~d\tilde\mu(x,y)\end{align*}
since the integrand is independent of $y$.
Because $\int\tilde\phi\circ \tilde f^n~d\tilde\mu= \int\tilde\phi~d\tilde\mu$, to complete the lemma it suffices to show that increasing $n$ makes
$$\left|\int\left(\inf_{y'\in I}\tilde\phi\circ \tilde f^n(x,y')\right)~d\tilde\mu(x,y)-\int\tilde\phi\circ \tilde f^n(x,y)~d\tilde\mu(x,y)\right|$$
arbitrarily small (this follows similarly when we replace $\inf$ by $\sup$).
Since $\tilde f$ is uniformly contracting on each $\xi_x$
by $(\star\star)$, and $\tilde\phi$ is uniformly continuous, for any $\eps>0$, for all large $n$,
$$\left|\left(\inf_{y'\in I}\tilde\phi\circ \tilde f^n(x,y')\right) -\tilde\phi\circ \tilde f^n(x,y)\right|<\eps,$$ for all $(x,y)\in \tilde I$.   Therefore,
$$\left|\int\left(\inf_{y'\in I}\tilde\phi\circ \tilde f^n(x,y')\right)~d\tilde\mu(x,y)-\int\tilde\phi\circ \tilde f^n(x,y)~d\tilde\mu(x,y)\right|<\eps.$$
We can also replace $\inf$ with $\sup$ here.  Hence $\tilde\nu= \tilde\mu$ as required.

Given an ergodic measure $\tilde\mu\in \M(\tilde f)$, it is clear that $\tilde p^{-1}(\tilde\mu)=\iota^{-1}(\tilde\mu)$ is ergodic.  Conversely, given an ergodic measure $\mu\in \M(f)$, the ergodicity of $\tilde p(\mu)$ follows as in \cite[Corollary 5.5]{ArPacPuVi}.
\end{proof}

\subsection{Relation between thermodynamics of systems on the interval, square and flow}

For a potential $\phi:I \to \R$, we define $\tilde\phi:\tilde I \to \R$ to be $\tilde\phi(x,y)=\phi(x)$.  Conversely, if $\tilde\phi: \tilde I \to \R$ is a potential depending only on the first coordinate then we define $\tilde\phi_1(x):=\tilde\phi(x,y)$.

\begin{lemma}
Given a potential $\phi:I \to \R$, $\mu$ is an equilibrium state for $(I,f,\phi)$ if and only if $\tilde\mu$ is an equilibrium state for $(\tilde I,\tilde f,\tilde\phi)$.
\label{lem:eq pushed}
\end{lemma}

\begin{proof}
Suppose that $\tilde\mu$ is an equilibrium state for $\tilde\phi$.  Then $$h(\tilde f, \tilde\mu)+\int\tilde\phi~d\tilde\mu=P(\tilde f, \tilde\phi).$$  We let $\mu$ be the projection of $\tilde\mu$ to $I$.  It is easy to show that $\int\tilde\phi~d\tilde\mu= \int \tilde\phi_1~d\mu$ and $h(\tilde f, \tilde\mu)= h(f, \mu)$. Then $$h(f, \mu)+\int\tilde\phi_1~d\mu=P(\tilde f, \tilde\phi).$$
As in \cite{Bowen}, $P(\tilde f, \tilde\phi)= P(f, \tilde\phi_1)$, so $\mu$ is an equilibrium state for $\tilde\phi_1$.

To complete the proof of the lemma, we observe that the above computations also imply that if $\mu\in \M(f)$ is an equilibrium state for $\phi$, then $\tilde\mu= \tilde p(\mu)$ is an equilibrium state for $\tilde\phi$.
\end{proof}

The following lemma gives us candidate equilibrium states for the Lorenz flow.

\begin{lemma}
For all $t\in (t^-, t^+)$, $\mu_t$, the equilibrium state for $\phi_t$ projects to a measure $\tilde\mu_t \in \M(\tilde f, r)$.
\label{lem:roof integr}
\end{lemma}

\begin{proof}
The fact that $\mu_t$ projects to an invariant measure $\tilde\mu_t\in \M(\tilde f)$ follows as in Lemma~\ref{lem:eq pushed}.  It remains to show that $\int r~d\tilde\mu<\infty$.
 By \eqref{eq:tempo}, $s(x)\asymp \log|x-c|$.  Therefore, since the critical point for $f:I \to I$ is non-flat, the integrability of $\log|Df|$ implies the integrability of $r$.  The fact that $\log|Df|\in L^1(\mu_t)$ is clear from the definitions of $\intclass$ and $\mu_t$.
\end{proof}

Next we extend to the flow.  Any $\tilde f$-invariant measure $\tilde\mu$ on $\tilde I$ can be identified with a $\check f$-invariant measure $\check\mu=\check\iota(\tilde\mu)$ where $\check\iota$ is defined in \eqref{eq:iota}.
Note by \cite[Corollary 5.10]{ArPacPuVi} if $\nu\in \M(\tilde f)$ is ergodic then $\check{\nu}\in \M(\check f)$ is ergodic.
We define the map
$$\hat p:\M(f) \overset{\tilde p}{\longrightarrow} \M(\tilde f) \overset{\check\iota}{\longrightarrow} \M(\check f) \overset{\check p}{\longrightarrow} \M(\hat f).$$
Given $\mu \in \M(f)$, let $\hat\mu:=\hat p(\mu)$.


\begin{lemma}
Suppose that $\hat\phi:\hat I \to \R$ gives a potential $\tilde\phi=\Delta_{\hat\phi}:\tilde I \to \R$
which depends only on the first coordinate. We set $\phi(x)=\tilde\phi(x,y)$ for any $y\in I$.  Then $\mu$ is an equilibrium state for $(I,f,\phi)$ if and only if $\hat\mu$ is an equilibrium state for $(\hat I,\hat f,\hat\phi)$.
\label{lem:eq flow}
\end{lemma}

\begin{proof}
This follows as in \cite[Theorem 4.4]{Cher}:
We may assume that $P(\phi)=0$.  So $$0=P(\phi)=h(\mu)+\int\phi~d\mu.$$ Then by Lemma~\ref{lem:eq pushed} the measure $\tilde\mu\in \M(\tilde f)$ has
$$0=P(\tilde\phi)=h(\tilde\mu)+\int\tilde\phi~d\tilde\mu.$$
We let $\hat\mu=\hat p(\tilde\mu)$.  Then by Lemma~\ref{lem:roof integr} and the Abramov formula,
$$h(\hat\mu)+\int\hat\phi~d\hat\mu =\frac{h(\tilde\mu)+\int\tilde\phi~d\tilde\mu}{\int r~d\tilde\mu}=0.$$  This computation also shows that $P(\hat\phi)=0$.  Hence $\hat\mu$ is an equilibrium state for $\hat\phi$.
The converse argument follows similarly.
\end{proof}

\begin{proposition}
For $\hat\phi_t$ and $\phi_t$ as in \eqref{eq:hat phi t} and \eqref{eq:phi t} respectively, $\mu\in \M(f)$ is an equilibrium state for $(I,f,\phi_t)$ if and only if $\hat\mu=\hat\iota(\mu)$ is an equilibrium state for $(\hat I,\hat f,\hat\phi_t)$.
\label{prop:nat eq flow}
\end{proposition}

\begin{proof}
Lemma~\ref{lem:eq flow} gives this immediately.
\end{proof}

\begin{proof}[Proof of Theorem~\ref{thm:eq for flow}]
Proposition~\ref{prop:nat eq flow} added to Theorem~\ref{thm:eq states} completes the proof.
\end{proof}

\subsection{Lyapunov spectrum for the flow: proof of Theorem~\ref{thm:lyap spec flow}}
\label{sec:LE flow}

To prove Theorem~\ref{thm:lyap spec flow}, first note that given $x\in I$ with $\lambda(x)=\alpha$, all points $(x,y)\in \xi_x$ must lie in $K(\alpha)$.  Moreover, $$\{\check f_s(x,y, 1): y\in \xi_x\text{ and } s\in [0,r(x,y, z))\} \subset K(\alpha).$$
Therefore, if we view $\hat f$ as a suspension flow over $(\tilde I, \tilde f)$ with roof function $r$,
$$\dim_H(K(\alpha))= \dim_H(K(\alpha)\cap \tilde I)+1= \dim_H(J(\alpha))+2=L(\alpha)+2.$$

Theorem~\ref{thm:lyap spec} then gives $L$ in terms of the pressure.

Therefore, to complete the proof of Theorem~\ref{thm:lyap spec flow} we need to check that the map from the suspension flow model to the flow $\hat f$ does not distort things too much; in particular is locally bilipschitz.  This allows us to assert the first equality above.
In \cite{PaldM} they refer to $(x,y,z)\in \hat I$ as a \emph{regular point} if there exists $(\tilde x, \tilde y, \tilde z)\in \tilde I$ and a neighbourhood $U_0$ such that $\hat f_s:U_0 \to U_s$ is a diffeomorphism where $U_s$ is a neighbourhood of $(\tilde x, \tilde y, \tilde z)$.  Note that any point in $\Lambda\cap \tilde I$ is regular.

For any regular point $(x, y, z)$, there a neighbourhood of $(\tilde x, \tilde y, \tilde z)$ in $\hat I$ such that the flow by $\hat f$ to the corresponding neighbourhood of $(x,y,z)$ is conjugated by $\check p$ to the parallel flow on a cube.  By the Tubular Flow Theorem of \cite[Chapter 2]{PaldM}, this conjugacy is bilipschitz in this neighbourhood. This proves Theorem~\ref{thm:lyap spec flow}.

\emph{Acknowledgements:} MT would like to thank J.M.\ Freitas for useful conversations.


\begin{thebibliography}{99}

\bibitem[Ab]{Abra} L.M.\ Abramov,
\emph{On the entropy of a flow,}
Dokl. Akad. Nauk SSSR {\bf 128} (1959) 873--875.

\bibitem[ABS]{Afr_etc} V.S.\ Afra\u{\i}movi\v{c}, V.V.\ Bykov, L.P.\ Sil'nikov,
\emph{The origin and structure of the Lorenz attractor,}
Dokl. Akad. Nauk SSSR \textbf{234} (1977) 336--339.

\bibitem[AK]{AmbKak} W.\ Ambrose, S.\ Kakutani,
\emph{Structure and continuity of measurable flows,}
 Duke Math. J. \textbf{9} (1942) 25--42.

\bibitem[AP]{ArPa} V.\ Ara\'ujo, M.J.\ Pacifico,
\emph{Three dimensional flows,}
Publica\c{c}\~oes Matem\'aticas do IMPA. [IMPA Mathematical Publications] 26$\sp {\rm o}$ Col\'oquio Brasileiro de Matem\'atica. [26th Brazilian Mathematics Colloquium] Instituto Nacional de Matem\'atica Pura e Aplicada (IMPA), Rio de Janeiro, 2007. 

\bibitem[APPV]{ArPacPuVi} V.\ Ara\'ujo, M.J.\ Pacifico, E.\ Pujals, M.\ Viana,
    \emph{Singular-hyperbolic attractors are chaotic,}
Trans. Amer. Math. Soc. \textbf{361} (2009) 2431-2485.

\bibitem[Ba]{Baladi} V.\  Baladi,
{\em Positive transfer operators and decay of correlations},
Advanced Series in Nonlinear Dynamics, {\bf 16} World Scientific
Publishing Co., Inc., River Edge, NJ, 2000.

\bibitem[BaI]{BarIom} L.\ Barreira, G.\ Iommi,
\emph{Suspension Flows Over Countable Markov Shifts,}
J. Stat. Phys. {\bf 124} (2006) 207--230.

\bibitem[BaS1]{BarSau_hyp_flow} L.\ Barreira, B.\ Saussol,
\emph{Multifractal analysis of hyperbolic flows,}
Comm. Math. Phys. \textbf{214} (2000) 339--371.

\bibitem[BaS2]{BarSau_var_flow} L.\ Barreira, B.\ Saussol,
\emph{Variational principles for hyperbolic flows,}  Differential equations and dynamical systems (Lisbon, 2000),  43--63, Fields Inst. Commun., 31, Amer. Math. Soc., Providence, RI, 2002.

\bibitem[Bo]{Bowen} R.\ Bowen, \emph{Equilibrium States and the Ergodic Theory of Anosov Diffeomorphisms,} Springer Lect. Notes in Math. 470 (1975).

\bibitem[BK]{BrKell} H.\ Bruin, G.\ Keller,  \emph{Equilibrium states for $S$-unimodal maps,} Ergodic Theory Dynam. Systems \textbf{18} (1998) 765--789.

\bibitem[BT1]{BTeqgen} H.\ Bruin, M.\ Todd, \emph{Equilibrium states for potentials with $\sup \phi-\inf \phi < h_{top}(f)$, }
    Comm. Math. Phys. \textbf{283} (2008) 579-611.

\bibitem[BT2]{BTeqnat} H.\ Bruin, M.\ Todd, \emph{Equilibrium states for interval maps: the potential $ -t \log|Df|$,}
    Ann. Sci. \'Ecole Norm. Sup. (4) \textbf{ 42} (2009) 559--600.

\bibitem[C]{Cher} N.\ Chernov,
\emph{Invariant measures for hyperbolic dynamical systems,}
Handbook of Dynamical Systems, Ed. by A. Katok and B. Hasselblatt,
Vol. 1A, pp. 321-407, North-Holland, Amsterdam, 2002.

\bibitem[Co]{Collexval} P.\ Collet,
{\em Statistics of closest return for some non-uniformly hyperbolic
systems,}
Ergodic Theory Dynam. Systems {\bf 21}  (2001) 401--420.

\bibitem[DHL]{DiHoLu}
K.\ D\'iaz-Ordaz, M.P.\ Holland, S.\ Luzzatto,
{\em Statistical properties of one-dimensional maps with critical points and singularities,}
Stoch. Dyn. {\bf 6} (2006) 423--458.

\bibitem[D1]{Dobthes} N.\ Dobbs,
\emph{Critical points, cusps and induced expansion in dimension one,}
Thesis, Universit\'e Paris-Sud, Orsay.

\bibitem[D2]{Dobcusp} N.\ Dobbs,
{\em On cusps and flat tops,}
Preprint (arXiv:0801.3815).

\bibitem[GaPa]{GalPac} S.\ Galatolo, M.J.\ Pacifico,
 \emph{Lorenz like flows: exponential decay of correlations for the Poincar\'e map, logarithm law, quantitative recurrence,}
Ergodic Theory and Dynamical Systems, to appear.

\bibitem[GPR]{GelRaPrz} K.\ Gelfert, F.\ Przytycki, M.\ Rams, \emph{Lyapunov spectrum for rational maps,} Preprint (arXiv:0809.3363).

\bibitem[GR]{GelRa} K.\ Gelfert, M.\ Rams,  \emph{The Lyapunov spectrum of some parabolic systems,} Ergodic Theory Dynam. Systems \textbf{29} (2009) 919-940.

\bibitem[GuW]{GuckWi} J.\ Guckenheimer, R.F.\ Williams,
\emph{Structural stability of Lorenz attractors,}
Inst. Hautes \'Etudes Sci. \textbf{50} (1979) 59--72.

\bibitem[HM]{HolMel} M.\ Holland, I.\ Melbourne,
\emph{Central limit theorems and invariance principles for Lorenz attractors,}  J. Lond. Math. Soc. (2)  \textbf{76}  (2007) 345--364.

\bibitem[IT1]{IomTeq} G.\ Iommi, M.\ Todd, \emph{Natural equilibrium states for multimodal maps,}
    Preprint (arXiv:0907.2406).

\bibitem[IT2]{IomTlyap} G.\ Iommi, M.\ Todd, \emph{Dimension theory for multimodal maps,}
    Preprint (arXiv:0911.3077).

\bibitem[K]{Kellbook} G.\ Keller,
{\em Equilibrium states in ergodic theory,} London Mathematical
Society Student Texts, 42. Cambridge University Press, Cambridge,
1998.

\bibitem[L]{Led} F.\ Ledrappier \emph{Some properties of absolutely continuous invariant measures on an interval,}
    Ergodic Theory Dynam. Systems \textbf{1} (1981) 77--93.

\bibitem[L]{Lorenz} E.N.\ Lorenz,
\emph{Deterministic nonperiodic flow,}
 J.Atmosph.Sci. \textbf{20} (1963).

\bibitem[Me]{Metz_SRB} R.J.\ Metzger,
\emph{Sinai-Ruelle-Bowen measures for contracting Lorenz maps and flows,}
 Ann. Inst. H. Poincar\'e Anal. Non Lin\'eaire, \textbf{17} (2000) 247--276.

\bibitem[MM]{MetzMor} R.J.\ Metzger, C.A.\ Morales,
 \emph{The Rovella attractor is a homoclinic class,} Bull. Braz. Math. Soc. (N.S.) \textbf{37} (2006) 89--101.

\bibitem[MPP]{MorPacPu} C.A.\ Morales, M.J.\ Pacifico, E.R.\ Pujals, \emph{Singular hyperbolic systems,}
  Proc. Amer. Math. Soc. \textbf{127} (1999) 3393--3401.

\bibitem[Na]{Nak} K.\ Nakaishi, \emph{Multifractal formalism for some parabolic maps,} Ergodic Theory Dynam. Systems \textbf{24} (2000) 843-857.

\bibitem[O]{Ols} L.\ Olsen, \emph{A multifractal formalism,}
Advances in Mathematics \textbf{116} (1995) 82-196.

\bibitem[PM]{PaldM} J.\ Palis, W.\ de Melo,
 \emph{Geometric theory of dynamical systems. An introduction,} Translated from the Portuguese by A. K. Manning. Springer-Verlag, New York-Berlin, 1982.

\bibitem[P]{Pesbook} Y.\ Pesin,  \emph{Dimension Theory in Dynamical Systems,} CUP, 1997.

\bibitem[PSa]{PesSad} Y.\ Pesin, V.\ Sadovskaya,
\emph{Multifractal analysis of conformal Axiom A flows,}
Comm. Math. Phys. \textbf{216} (2001) 277--312.

\bibitem[PSe]{PeSe} Y.\ Pesin, S.\ Senti, \emph{Equilibrium measures for Maps with inducing schemes}  J. Mod. Dyn. {\bf 2} (2008) 1--31.

\bibitem[PolW]{PolWei}  M.\ Pollicott, H.\ Weiss, \emph{Multifractal analysis of Lyapunov exponent for continued fraction and Manneville-Pomeau transformations and applications to Diophantine approximation,}  Comm. Math. Phys. \textbf{207}  (1999) 145--171.

\bibitem[Pr]{Prz} F.\ Przytycki, {\em Lyapunov characteristic exponents are nonnegative,} Proc. Amer. Math. Soc. {\bf 119} (1993) 309--317.

\bibitem[PrR]{PrzR-L} F.\ Przytycki, J.\ Rivera-Letelier, \emph{Nice inducing schemes and the thermodynamics of rational maps,} arXiv:0806.4385

\bibitem[Ro]{Rovella} A.\ Rovella,
\emph{The dynamics of perturbations of the contracting Lorenz attractor,}
Bol. Soc. Brasil. Mat. (N.S.) \textbf{24} (1993) 233--259.

\bibitem[Ru]{Ruelle} D.\ Ruelle,
{\em Thermodynamic formalism,} Addison Wesley, Reading MA, 1978.

\bibitem[Si]{Sinai}  Y.\ Sinai,
{\em Gibbs measures in ergodic theory,}  Uspehi Mat.
Nauk { \bf 27}  (1972) 21--64.

\bibitem[Sp]{Sp} C.\ Sparrow,
 \emph{The Lorenz equations,}  Chaos,  111--134, Nonlinear Sci. Theory Appl., Manchester Univ. Press, Manchester, 1986.

\bibitem[Tu]{Tucker} W.\ Tucker,
\emph{A rigorous ODE solver and Smale's 14th problem,}
Found. Comput. Math. \textbf{ 2}  (2002) 53--117.

\bibitem[U]{Urb_surv} M.\ Urba\'nski,
\emph{Measures and dimensions in conformal dynamics,}
  Bull. Amer. Math. Soc. (N.S.) \textbf{ 40}  (2003) 281--321.

\bibitem[V]{Vibook}  M.\ Viana,
{\em Stochastic Dynamics of Deterministic Systems,}
21$^o$ Col\'oquio Brasileiro de Matem\'atica, IMPA, 1997.

\bibitem[W]{Wei} H.\ Weiss, \emph{The Lyapunov spectrum for conformal expanding maps and Axiom A surface diffeomorphisms,}
    J. Statist. Phys. \textbf{95} (1999) 615-632.

\end{thebibliography}
\end{document}